\documentclass[12pt,a4paper]{article}  
\usepackage{hyperref}
\usepackage{t1enc}
\usepackage[latin1]{inputenc}
\usepackage[english]{babel}
\usepackage{graphicx}
\usepackage{amsmath}
\usepackage{amssymb}
\usepackage{amsthm}
\font\bBB=msbm10

\def\bBR{\mbox{\bBB R}}

\def\bBZ{\mbox{\bBB Z}}

\def\R3{\bBR ^3}
\def\Rn{\bBR ^n}
\def\Z{\bBZ}
\def\Zz{\bBZ_2}
\newcommand*{\defeq}{\mathrel{\vcenter{\baselineskip0.5ex \lineskiplimit0pt
                     \hbox{\scriptsize.}\hbox{\scriptsize.}}}%
                     =}
\DeclareMathOperator{\Tr}{Tr}

\begin{document}
\title {The Fundamental Group of $SO(n)$ Via Quotients of Braid Groups}
\author{
  Ina Hajdini\thanks{American University in Bulgaria, 2700 Blagoevgrad, Bulgaria; current affiliation: Drexel University, Philadelphia, PA 19104.} \hspace{0.07cm} and Orlin Stoytchev\thanks{American University in Bulgaria, 2700 Blagoevgrad, Bulgaria.}}
	
\maketitle

\begin{abstract} 
We describe an algebraic proof of the well-known topological fact that $\pi_1(SO(n)) \cong \bBZ/2\bBZ$. The fundamental group of $SO(n)$ appears in our approach as the center of a certain finite group defined by generators and relations. The latter is a factor group of the braid group $B_n$, obtained by imposing one additional relation and turns out to be a nontrivial central extension by $\bBZ/2\bBZ$ of the corresponding  group of rotational symmetries of the hyperoctahedron in dimension $n$. 
\end{abstract}

\section{Introduction.}

The set of all rotations in $\Rn$ forms a group denoted by $SO(n)$. We may think of it as the group of $n\times n$ orthogonal matrices with unit determinant. As a topological space it has the structure of a smooth $(n(n-1)/2)$-dimensional submanifold of $\bBR^{n^2}$. The group structure is compatible with the smooth one in the sense that the group operations are smooth maps, so it is a Lie group. The space $SO(n)$ when $n\ge 3$ has a fascinating topological property---there exist closed paths in it (starting and ending at the identity) that cannot be continuously deformed to the trivial (constant) path, but going twice along such a path gives another path, which is deformable to the trivial one. For example, if you rotate an object in $\R3$ by $2\pi$ along some axis, you get a motion that is not deformable to the trivial motion (i.e., no motion at all), but a rotation by $4\pi$ is deformable to the trivial motion. Further, a rotation by $2\pi$ along any axis can be deformed to a rotation by $2\pi$ along any other axis. We shall call a {\it full rotation} in $\Rn$ any motion that corresponds to a closed path in $SO(n)$, starting and ending at the identity. Thus, it turns out that there are two classes of full rotations: topologically trivial, i.e., deformable to the trivial motion, and topologically nontrivial. Every nontrivial full rotation can be deformed to any other nontrivial full rotation. Two consecutive nontrivial full rotations produce a trivial one. \par
For any topological space, one can consider the set of closed paths starting and ending at some fixed point, called {\it base-point}. Two closed paths that can be continuously deformed to each other, keeping the base-point fixed, are called \textit{homotopic}. One can multiply closed paths by {\it concatenation}, i.e., take the path obtained (after appropriate reparametrization) by traveling along the first and then along the second. There is also an inverse for each path--the path traveled in reverse direction. These operations turn the set of homotopy classes of closed paths (with a given base-point) into a group and it is an important topological invariant of any topological space. It was introduced by Poincar\'{e} and is called the {\it first homotopy group} or the \textit{fundamental group} of the space, denoted by $\pi_1$. Thus, the property of $SO(n)$ stated above is written concisely as $\pi_1(SO(n)) \cong \bBZ/2\bBZ\equiv \Zz$. \par
This specific topological property in the case $n=3$ plays a fundamental role in our physical world. To the two homotopy classes of closed paths in $SO(3)$ correspond precisely two principally different types of elementary particles: bosons, with integer spin, and fermions, with half-integer spin, having very distinct physical properties. The difference can be traced to the fact that the complex (possibly multicomponent) wave function determining the quantum state of a boson is left unchanged by a rotation by $2\pi$ of the coordinate system while the same transformation multiplies the wave function of a fermion by $-1$. This is possible since only the modulus of the wave function has a direct physical meaning, so measurable quantities are left invariant under a full rotation by $2\pi$. However, as discovered by Pauli and Dirac, one needs to use wave functions having this (unexpected) transformation property for the correct description of particles with half-integer spin, such as the electron. The careful analysis showed (\cite{Wig,Barg}) that the wave function has to transform properly only under transformations which are in a small neighborhood of the identity. A "large" transformation such as a rotation by $2\pi$ can be obtained as a product of small transformations, but the transformed wave function need not come back to itself--there may be a complex phase multiplying it. From continuity requirements it follows that if one takes a closed path in $SO(3)$ which is contractible, the end-point wave function must coincide with the initial one. Therefore, a rotation by $4\pi$ should bring back the wave function to its initial value and so the phase factor corresponding to a $2\pi$-rotation can only be $-1$. What we have just described is the idea of the so-called {\it projective} representations of a Lie group, which we have to use in quantum physics. As we see on the example of $SO(3)$, they exist because the latter is not simply-connected, i.e., $\pi_1(SO(3))$ is not trivial. Projective representations of a non-simply-connected Lie group are in fact representations of its covering group. In the case of $SO(3)$ this is the group $SU(2)$ of 
$2\times 2$ unitary matrices with unit determinant. Topologically, this is the three-dimensional sphere $S^3$; it is a double cover of $SO(3)$ and the two groups are locally isomorphic and have the same Lie algebra.
\par 
The standard proof that $\pi_1(SO(n)) \cong \bBZ/2\bBZ$ when $n=3$ uses substantially Lie theory. A $2-1$ homomorphism $SU(2)\rightarrow SO(3)$ is exhibited, which is a local isomorphism of Lie groups. This is the double covering map in question, sending any two antipodal points of $SU(2)$ (i.e., $S^3$) to one point in $SO(3)$. The case $n>3$ reduces to the above result by applying powerful techniques from homotopy theory.
\par
There are several more or less easy geometric methods to unveil the nontrivial topology of $SO(3)$. Among them, a well-known demonstration is the so-called "Dirac's belt trick" in which one end of a belt is fastened, the other (the buckle) is rotated by $4\pi$. Then without changing the orientation of the buckle, the belt is untwisted by passing it around the buckle (see, e.g., \cite{Egan, Palais} for nice Java applets and explanations). A refinement of "Dirac's belt trick" was proposed in \cite{Sto} where an isomorphism was constructed between homotopy classes of closed paths in $SO(3)$ and a certain factor group of the group $P_3$ of pure braids with three strands. This factor group turns out to be isomorphic to $\bBZ/2\bBZ$. The idea is fairly simple and is based on the following experiment: attach the ends of three strands to a ball, attach their other ends to the desk, perform an arbitrary number of full rotations of the ball to obtain a plaited braid. Then try to unplait it without further rotating the ball. As expected, braids that correspond to contractible paths in $SO(3)$ are trivial, while those corresponding to noncontractible paths form a single nontrivial class. \par
While the method of \cite{Sto} is simple and easy to visualize, it has the disadvantage that it does not lend itself to a generalization to higher dimensions. (A geometric braid in $\Rn$ is always trivial when $n>3$.)
The present paper takes a different, more algebraic approach. We study a certain discrete (in fact finite) group of homotopy classes of paths in $SO(n)$, starting at the identity and ending at points which are elements of some fixed finite subgroup of $SO(n)$. It turns out that it is convenient to use the finite group of rotational symmetries of the hyperoctahedron (the polytope in dimension $n$ with vertices $\{(\pm 1,0,...0), (0,\pm 1,...,0),...,(0,0,...,\pm 1)\}$). Each homotopy class contains an element consisting of a chain of rotations by $\pi/2$ in different coordinate planes. These simple motions play the role of generators of our group. Certain closed paths obtained in this way remain in a small neighborhood of the identity (in an appropriate sense, explained later) and can be shown explicitly to be contractible. Thus, certain products among the generators must be set to the identity and we get a group  defined by a set of generators and relations. Interestingly, the number of (independent) generators is $n-1$ and the relations, apart  from one of them, are exactly Artin's  relations for the braid group $B_n$. In this way we obtain for each $n$ a finite group, which is the quotient of  $B_n$ by the normal closure of the group generated by the additional relation. When $n=3$ the order of the group turns out to be 48 and it is the so-called {\it binary octahedral group} which is a nontrivial extension by $\bBZ/2\bBZ$ of the group of rotational symmetries of the octahedron. We may think of the former as a double cover of the latter and this is a finite version of the double cover $SU(2)\rightarrow SO(3)$. The next groups in the series have orders 384, 3840, 46080, etc.; in fact the order is given by $2^nn!$. Note that these groups have the same orders as the Coxeter groups of all symmetries (including reflections as well as rotations) of the respective hyperoctahedra, but they are different. This is analogous to the relationship between $O(n)$ and $SO(n)$ on the one hand, and $Spin(n)$ and $SO(n)$ on the other.\par
It turns out that the subgroup of homotopy classes of closed paths, i.e. $\pi_1(SO(n))$, in each case either coincides or lies in  the center of the respective group and is isomorphic to $\bBZ/2\bBZ$. Factoring by it, the respective rotational hyperoctahedral groups are obtained. We believe that these are new presentations of all rotational hyperoctahedral groups and their double covers. \\

\section{Groups of Homotopy Classes of Paths.}
We will consider paths in $SO(n)$ starting at the identity. In other words, we have continuous functions 
$R:[0,1]\rightarrow SO(n)$, subject to the restriction $R(0)=Id$.
Because the target space is a group, there is a natural product of such paths, i.e. if $R_1$ and $R_2$ are paths of this type, we first translate $R_1$ by the constant element $R_2(1)$ to the path  $R_1R_2(1)$. Then we concatenate $R_2$ with $R_1R_2(1)$. Thus, by $R_1R_2$ we mean the path: \\
\begin{equation} 
\text{$(R_1R_2)(t)$}=\begin{cases}
\text{$R_2(2t)$}, & \text{if $t<{1\over 2}$} \\
\text{$R_1(2t-1)R_2(1)$}, & \text{if $t\ge{1\over 2}$} \\
\end{cases}\nonumber
\end{equation}
For each $R$ we denote by $R^{-1}$ the path given by
$$R^{-1}(t)=(R(t))^{-1}\ .$$
The set formed by continuous paths in $SO(n)$ obviously contains an identity, which is the constant path. The product in this set is not associative since  $(R_1R_2)R_3$ and $R_1(R_2R_3)$ are different paths (due to the way we parametrize them), even though they trace the same curve. Also, $R^{-1}$ is by no means the inverse of $R$. In fact, there are no inverses in this set. However, the set of homotopy classes of such paths (with fixed ends) is a group with respect to the induced operations. 
Further, we will use $R$ to denote the homotopy class of $R$. Now $R^{-1}$ is the inverse of $R$. 
Recall that in algebraic topology one constructs the universal covering space of a non-simply connected space by taking the homotopy classes of paths starting at some fixed point (base-point). Within a given homotopy class, what is important is just the end-point and so the covering space is locally homeomorphic to the initial one. However, paths with the same end-point belonging to different classes are different points in the covering space. Thus, one effectively "unwraps" the initial space. We see that the group we have defined is none other than the universal covering group of $SO(n)$. In the case $n=3$ this is the group $SU(2)$; in general it is the group denoted by $Spin(n)$. \par
Our aim is of course to show that the subgroup corresponding to closed paths is $\bBZ_2$ or equivalently that the covering map, which is obtained by taking the end-point of a representative path, is $2\rightarrow 1$. Since the full group of homotopy classes of paths starting at the identity is uncountable and difficult to handle, the main idea of this paper is to study a suitable discrete, in fact finite, subgroup by limiting the end-points to be elements of the group of rotational symmetries of the hyperoctahedron in dimension $n$. In what follows we denote this group by $G$. (One could perhaps take any large enough finite subgroup of $SO(n)$, but using the hyperoctahedral group seems the simplest.)\par
The further study of $G$ requires algebra and just a couple of easy results from analysis and geometry. We state these here and leave the proofs for the appendix.\par
By a {\it generating path} $R_{ij}$ we will mean a rotation from $0$ to $\pi/2$ in the coordinate plane $ij$.
\newtheorem{lem}{Lemma}
\begin{lem}
Every homotopy class of paths in $SO(n)$ starting at the identity and ending at an element of the rotational hyperoctahedral group contains a representative which is a product of generating paths. 
\end{lem}
Each vertex of the hyperoctahedron lies on a coordinate axis---either in the positive or negative direction --- and determines a closed half-space (of all points having the respective coordinate nonnegative or nonpositive) to which it belongs. Let us call {\it local closed paths} those closed paths in $SO(n)$ for which no vertex of the hyperoctahedron leaves the closed half-space to which it belongs initially.
\begin{lem}
Local closed paths are contractible. A closed path consisting of generating paths is contractible if and only if the word representing it can be reduced to the identity by inserting expressions describing local closed paths.
\end{lem}
\section{The case $\boldsymbol{n=3}$}

We consider three-dimensional rotations separately since the essential algebraic properties of $G$ are present already here and in a sense the higher-dimensional cases are a straightforward generalization. 
It is worth recalling some facts about the group of rotational symmetries of the octahedron in three-dimensions. The octahedron is one of the five regular convex polyhedra, known as Platonic solids. It has six vertices and eight faces which are identical equilateral triangles. Each vertex is connected by an edge to all other vertices except the opposite one. We may assume that the vertices lie two by two on the three coordinate axes. We will enumerate the vertices from 1 to 6 as follows: $1=(1, 0, 0)$, $2=(0, 1, 0)$, $3=(0, 0, 1)$, $4=(0,0,-1)$, $5=(0,-1,0)$ and $6=(-1,0,0)$. The octahedron is the dual polyhedron of the cube and so they have the same symmetry group. The analogous statement is true in any dimension. We find it convenient to think of a spherical model of the octahedron --- the edges connecting the vertices are parts of large circles on the unit sphere (Figure \ref{octah}).\\
\begin{figure}
\centering
\vskip -10mm
\includegraphics[width=60mm]{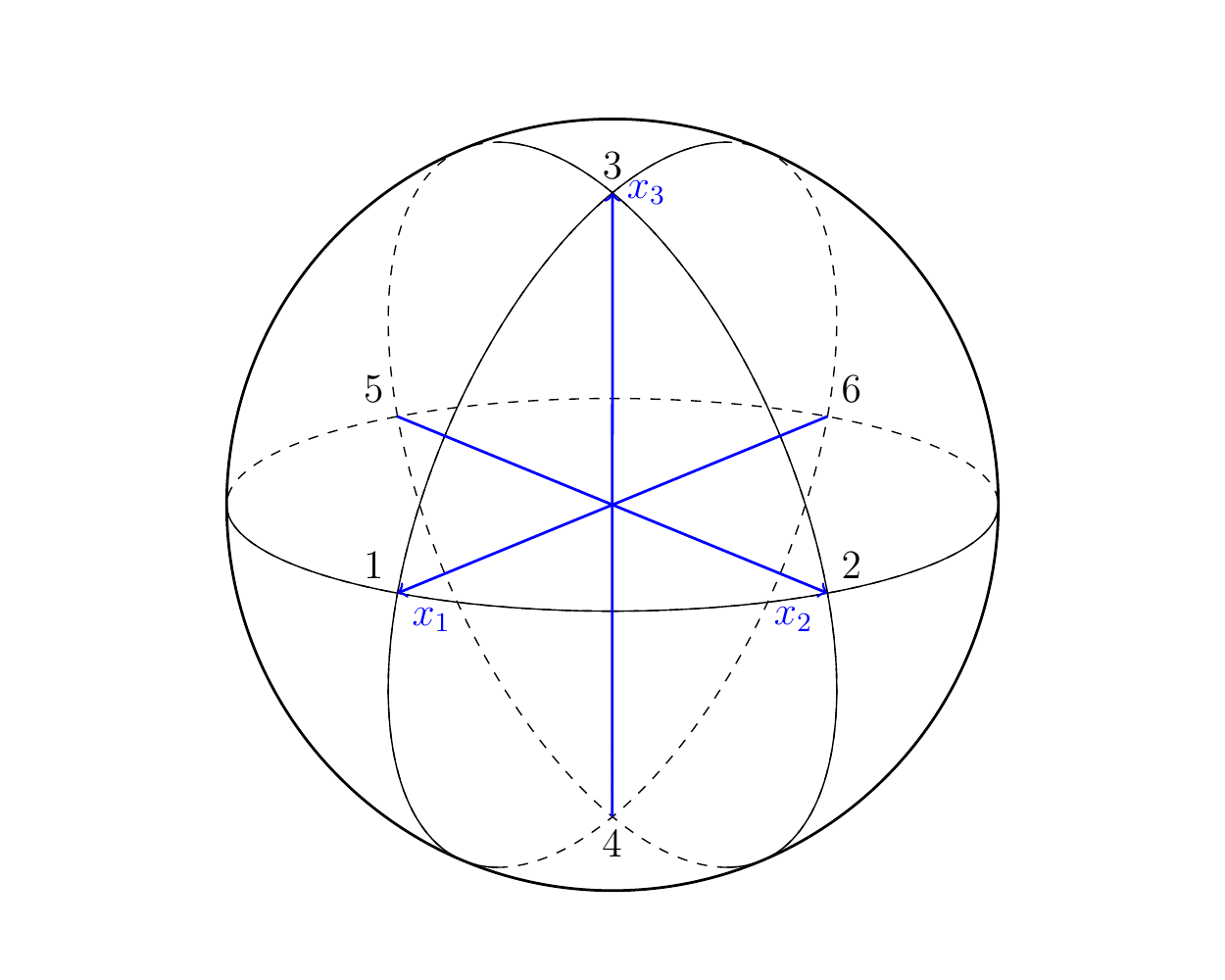}
\caption{A spherical octahedron with its three axes of rotational symmetries of order 4} \label{octah}
\end{figure}
The octahedron has three types of rotational symmetries belonging to cyclic subgroups of different orders:
\\
\emph{Type 1}: Three cyclic subgroups of order $4$ generated by rotations by $\pi/2$  around the three coordinate axes. Each such symmetry preserves the two vertices  lying on that axis while permuting the rest. \\
\emph{Type 2}: Six cyclic subgroups of order $2$, generated by rotations by $\pi$ around axes connecting the centers of six pairs of opposite edges.
 \\
\emph{Type 3}: Four cyclic subgroups of order $3$ generated by rotations by $2\pi/3$ around axes connecting the centers of four pairs of opposite faces.\\
Counting the number of nontrivial elements in each subgroup and adding $1$ shows that the group of rotational symmetries of the octahedron has order $24$. As each symmetry permutes the six vertices, it is obviously a subgroup of the symmetric group $S_6$. It is well-known and easy to see experimentally that the group of rotational symmetries of the octahedron is faithfully represented as the group of permutations of the four pairs of opposite faces. Therefore, it is isomorphic to $S_4$. All symmetric groups can be realized as finite reflection groups and thus finite Coxeter groups of type $A_n$ \cite{Humph}. We have $S_n\cong A_{n-1}$ and in the case at hand the group is $A_3$ --- the full symmetry group (including reflections) of the tetrahedron. The case $n=3$ is an exception. When $n>3$ the respective rotational hyperoctahedral group is not a Coxeter group. For any $n$ however, it is a normal subgroup of index $2$ of the respective full hyperoctahedral group which is a Coxeter group of type $B_n$ (not to be confused with the braid group on $n$ strands, for which the same notation is used).\par
\begin{figure}[h!]
\vskip-10mm
\centering
\includegraphics[width=100mm]{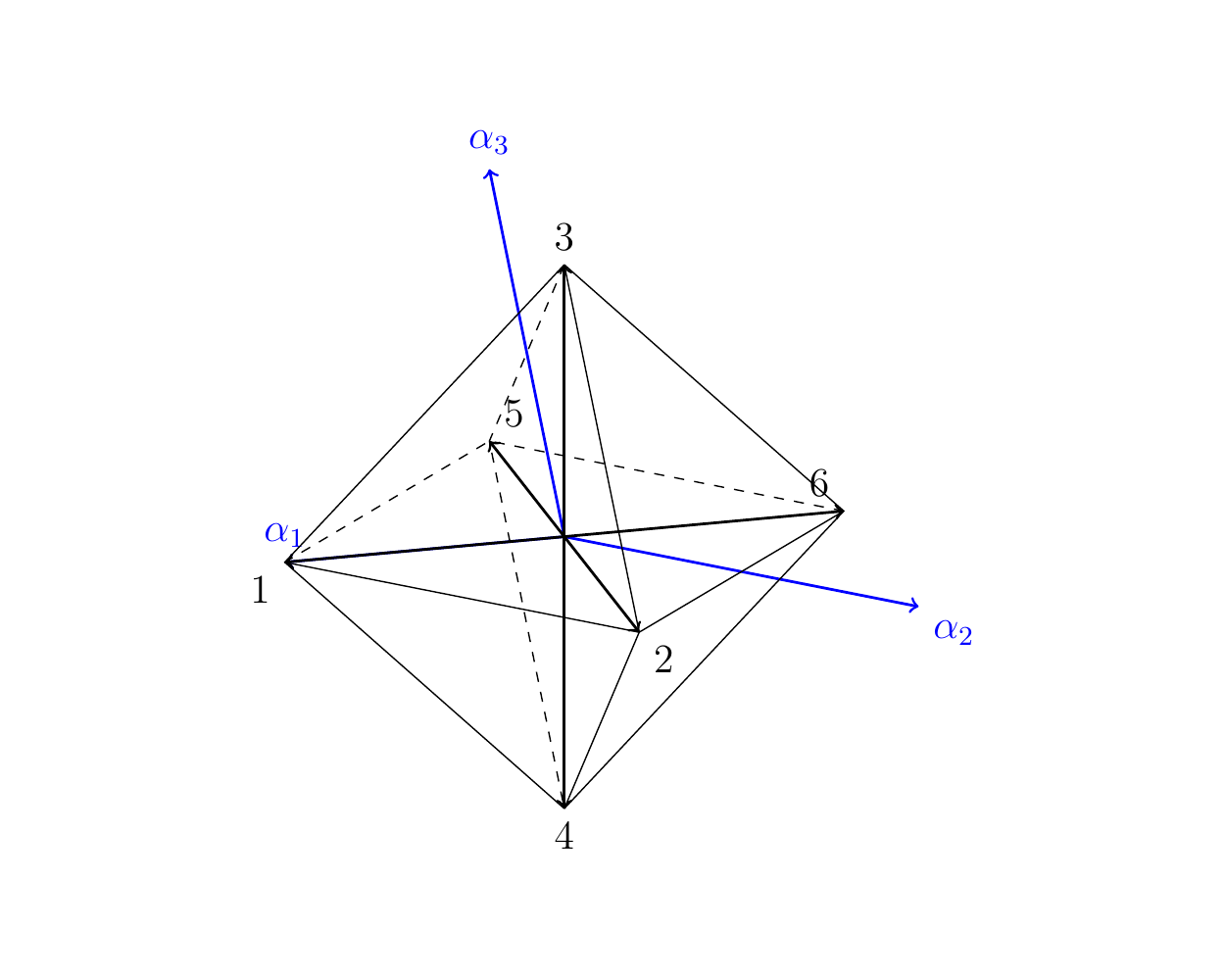}
\vskip-10mm
\caption{Three simple roots for $B_3$} \label{roots}
\end{figure}
For pedagogical purposes we choose to make a short description of the full hyperoctahedral group as a symmetry group of the hyperoctahedron, which is generated by reflections. Starting with the octahedron in dimension 3, we see that it is preserved by reflections with respect to planes perpendicular to the vectors 
$(\pm1,0,0),\,(0,\pm1,0),\,(0,0,\pm1)$ (the six vertices) and $(\pm1,\pm1,0),\\\,(\pm1,0,\pm1),\,(0,\pm1,\pm1)$ (the middles of the twelve edges). The vectors just listed are called {\it roots} and they satisfy two essential properties which are axioms for a {\it root system}: 1) Each reflection determined by a root maps the root system onto itself, and 2) The root system contains together with each root its negative but no other multiples of that root. Among the 18 roots, we can choose (not uniquely) 3 roots, called {\it simple roots} with the following properties: 3) They span the whole space in which the root system lives, and 4) Each root is either entirely positive linear combination of the simple ones or is entirely negative linear combination of those. One possible choice of simple roots is (see Figure \ref{roots})
$$\boldsymbol{\alpha}_1=(1,0,0),\quad\boldsymbol{\alpha}_2=(-1,1,0),\quad\boldsymbol{\alpha}_3=(0,-1,1)\,.$$
One can check that indeed every root is a linear combination of these three with either entirely positive or entirely negative coefficients. For example $(0,1,1)=2\boldsymbol{\alpha}_1+2\boldsymbol{\alpha}_2+\boldsymbol{\alpha}_3$. Notice that the angle between $\boldsymbol{\alpha}_1$ and $\boldsymbol{\alpha}_2$ is $135^{\circ}$ or $3\pi/4$, the angle between $\boldsymbol{\alpha}_2$ and $\boldsymbol{\alpha}_3$ is $120^{\circ}$ or $2\pi/3$, while the angle between $\boldsymbol{\alpha}_1$ and $\boldsymbol{\alpha}_3$ is $90^{\circ}$ or $\pi/2$. This is always the case --- any two simple roots form either obtuse angle or are orthogonal. The angle between two roots is further restricted by the fact that the product of two reflections corresponding to two different roots is a rotation at an angle linked to the angle between the roots. More precisely, if the angle between $\boldsymbol{\alpha}_i$ and $\boldsymbol{\alpha}_j$ is $\pi-\theta$ and the corresponding reflections are denoted as $s_{\alpha_i}$ and $s_{\alpha_j}$, then $s_{\alpha_i}s_{\alpha_j}$ is a rotation by $2\theta$. Since this rotation must have some finite order $m$, we have $\theta={\pi\over m}$. The whole finite reflection group is determined by the positive integers $m(\boldsymbol{\alpha}_i,\boldsymbol{\alpha}_j)$ for each pair of simple roots. Note that because all $\boldsymbol{\alpha}_i$ are reflections and thus have order two, the diagonal entries $m(\boldsymbol{\alpha}_i,\boldsymbol{\alpha}_i)=1$ always. For the octahedron from the angles between simple roots we get  $m(\boldsymbol{\alpha}_1,\boldsymbol{\alpha}_2)=4$, $m(\boldsymbol{\alpha}_2,\boldsymbol{\alpha}_3)=3$, and $m(\boldsymbol{\alpha}_1,\boldsymbol{\alpha}_3)=2$. The information is traditionally encoded in the so-called {\it Coxeter -- Dynkin diagram} where each simple root is represented by a node, two nodes are connected by an edge if $m>2$ and $m$ is written as a label below the corresponding edge if $m>3$. The Coxeter -- Dynkin diagram for the octahedral group is shown on Figure \ref{Dynkin}.\par
\begin{figure}
\centering
\includegraphics[width=30mm]{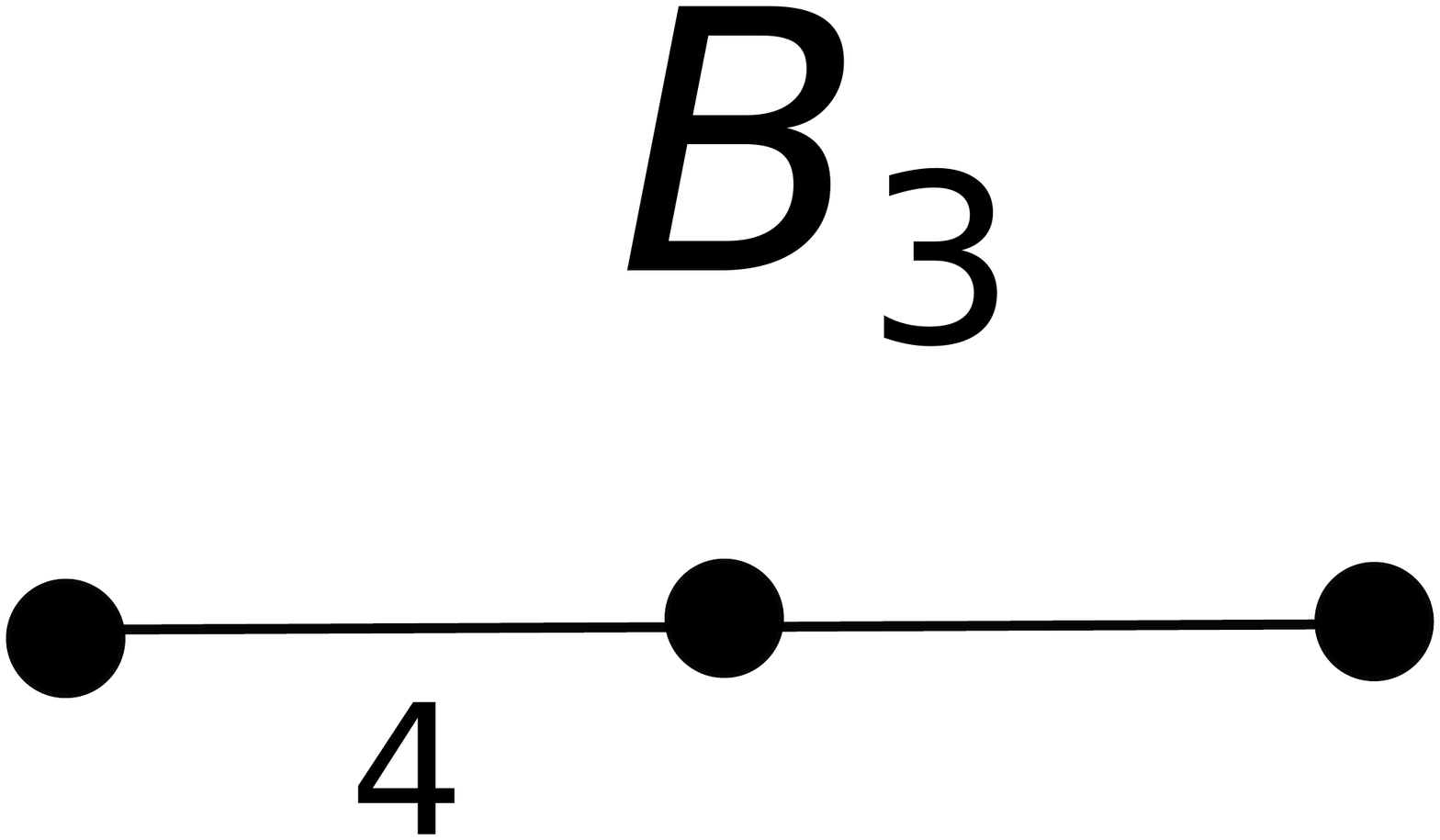}\quad\quad
\includegraphics[width=30mm]{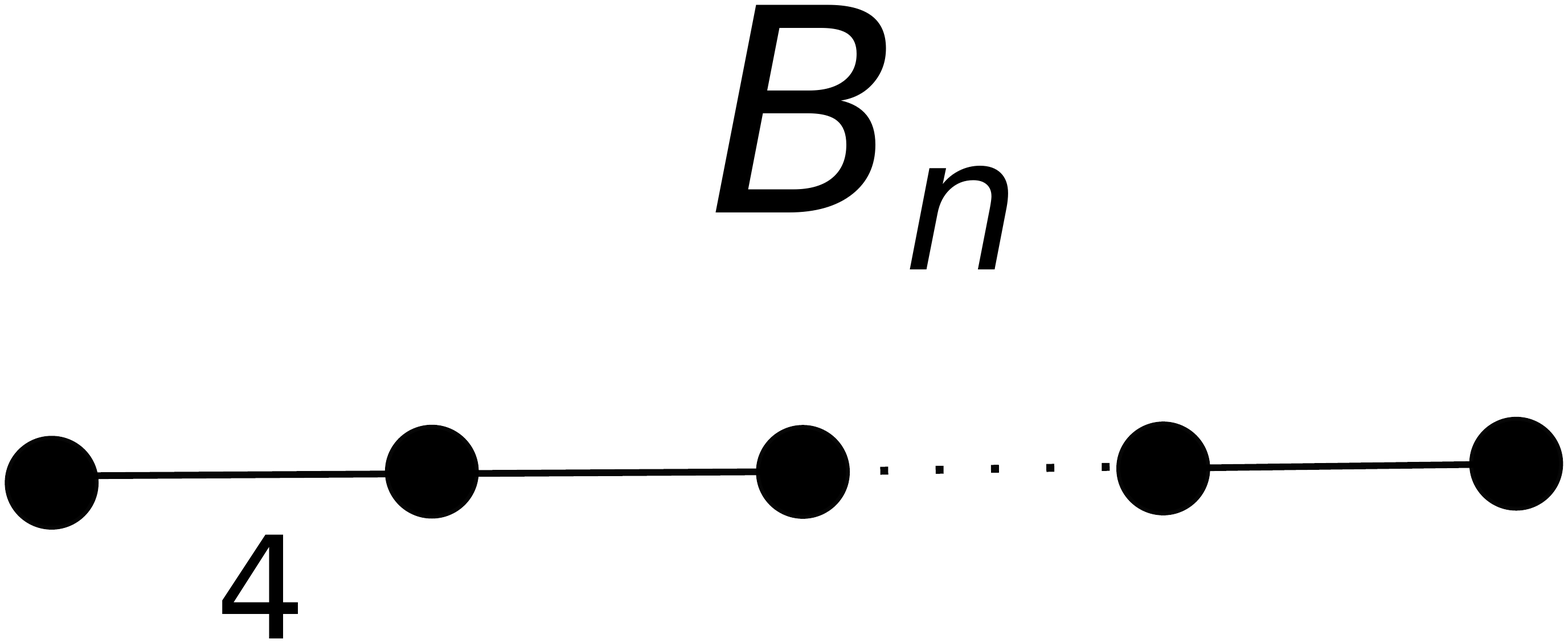}
\caption{The Coxeter-Dynkin diagrams for $B_3$ and $B_n$} \label{Dynkin}
\end{figure}
The generalization to higher $n$ is pretty straightforward. The roots will be all $2n$ vertices and the middles of all $2n(n-1)$ edges of the respective hyperoctahedron. It is easy to figure out that one can choose the following $n$ simple roots: $(1,0,...,0),\,(-1,1,0,...,0),...,(0,...,0,-1,1)$, and calculating the integers $m$, we see that the group has a Coxeter -- Dynkin diagram as in Figure \ref{Dynkin}. The groups of this type are denoted by $B_n$, where $n$ is the number of simple roots and is called the {\it rank}.\par
Alternatively, the full hyperoctahedral group can be thought of as the {\it wreath product} $S_2\wr S_n$. The wreath product in this special case is the semidirect product of the product of $n$ copies of $S_2$ with $S_n$, where $S_n$ acts on the first factor by permuting its components. More precisely, let $\Sigma=S_2\times S_2\times \cdots\times S_2$ and let $\sigma=(\sigma_1,\sigma_2,...,\sigma_n)\in \Sigma$. The symmetric group $S_n$, considered as a permutation group of $(1,2,...,n)$, acts naturally by automorphisms on $\Sigma$, namely, if $h\in S_n$
$$(h\sigma)_i=\sigma_{h^{-1}(i)}\,.$$
Now, the wreath product $S_2\wr S_n$ is just the semidirect product $\Sigma\rtimes S_n$.
The $n$ pairs of opposite vertices determine $n$ mutually orthogonal (non-oriented) lines in $\Rn$. A rotation in the $ij$th plane by $\pi/2$ permutes the $i$th and $j$th lines. An arbitrary symmetry can be realized by an arbitrary permutation of the lines plus possible reflections with respect to the hyperplanes perpendicular to the $n$ lines. This explains the structure of the full hyperoctahedral group as a wreath product. One advantage is that it is easy to see that the order of the group is $2^nn!$.\par
The rotational hyperoctahedral group forms a normal subgroup of index 2 in the full hyperoctahedral group. In terms of orthogonal matrices, this is the subgroup of matrices with unit determinant. Note that each rotation is a product of two reflections.
Coming back to the rotational octahedral group (i.e. the case $n=3$), we observe that it is generated by three elements, denoted by $r_1$, $r_2$ and $r_3$, 
of \emph{Type 1} --- rotations by $\pi/2$ around the three coordinate axes.  They permute the vertices of the octahedron and using standard notations for permutations in terms of cycles (omitting the trivial 1-cycles) we can write: 
$$r_1= (2354),\quad r_2= (1463),\quad r_3= (1265).$$
Obviously, these three elements have order 4, and generate the respective cyclic groups of \emph{Type 1} but they also generate all symmetries of \emph{Type 2} and \emph{Type 3}: \par
The six rotations by $\pi$ around the respective rotational axis (as described above) in terms of $r_1$, $r_2$ and $r_3$ are $r_2r_3r_1r_3^{\,2}$,  $r_2r_3r_1$,  $r_3r_1r_2r_1^{\,2}$,  $r_3r_1r_2$,  $r_1r_2r_3r_2^{\,2}$, and $r_1r_2r_3$. \par
The eight rotations by $\pm2\pi/3$ around the respective rotational axis (as described above) in terms of $r_1$, $r_2$ and $r_3$ are $r_1r_2$, $r_1r_3$, $r_1r_2^{\,3}$, $r_1r_3^{\,3}$, and their inverses. \par
We consider the discrete group $G$ generated by the generating paths $R_1:=R_{23}$, $R_2:=R_{31}$, and $R_3:=R_{12}$, treated as homotopy classes. (We have $R_i(1)=r_i,\ i=1,2,3$.) Local closed paths built out of the generators $R_i$ and their inverses must be set to identity. Apart from trivial cases where an $R_i$ is followed by its inverse, we have a family of paths for which each vertex either goes around the edges of a single triangular face  of the octahedron or moves along an edge and comes back. Inspecting all possible ways in which such "triangular" closed paths can be built, we see that each one is represented by a word of four letters. Each word contains either $R_i$ or $R_i^{-1}$ for each $i$. No word contains twice a given letter but it may contain a letter together with its inverse. In that case this letter and its inverse conjugate one of the other two letters. Here is a list of all identities that follow:
\begin{eqnarray}\label{Id}
1&=&R_3^{-1}R_1^{-1}R_2^{-1}R_1=
R_3R_1^{-1}R_2R_1=
R_2R_3^{-1}R_2^{-1}R_1=
R_2^{-1}R_3R_2R_1\\\nonumber
&=&R_3^{-1}R_2^{-1}R_3R_1=
R_3R_2R_3^{-1}R_1=
R_2^{-1}R_1^{-1}R_3R_1=
R_2R_1^{-1}R_3^{-1}R_1\\\nonumber
&=&R_1R_3R_1^{-1}R_2=
R_1^{-1}R_3^{-1}R_1R_2=
R_3R_1^{-1}R_3^{-1}R_2=
R_3^{-1}R_1R_3R_2 \\\nonumber
&=&R_3^{-1}R_2^{-1}R_1R_2=
R_3R_2^{-1}R_1^{-1}R_2=
R_1R_2^{-1}R_3R_2=
R_1^{-1}R_2^{-1}R_3^{-1}R_2\\\nonumber
&=&R_1^{-1}R_2R_1R_3=
R_1R_2^{-1}R_1^{-1}R_3=
R_2R_3^{-1}R_1R_3=
R_2^{-1}R_3^{-1}R_1^{-1}R_3 \\\nonumber
&=&R_2^{-1}R_1^{-1}R_2R_3=
R_2R_1R_2^{-1}R_3=
R_1^{-1}R_3^{-1}R_2R_3=
R_1R_3^{-1}R_2^{-1}R_3. \\\nonumber
\end{eqnarray}
\begin{figure}[h!]
\centering
\includegraphics[width=80mm]{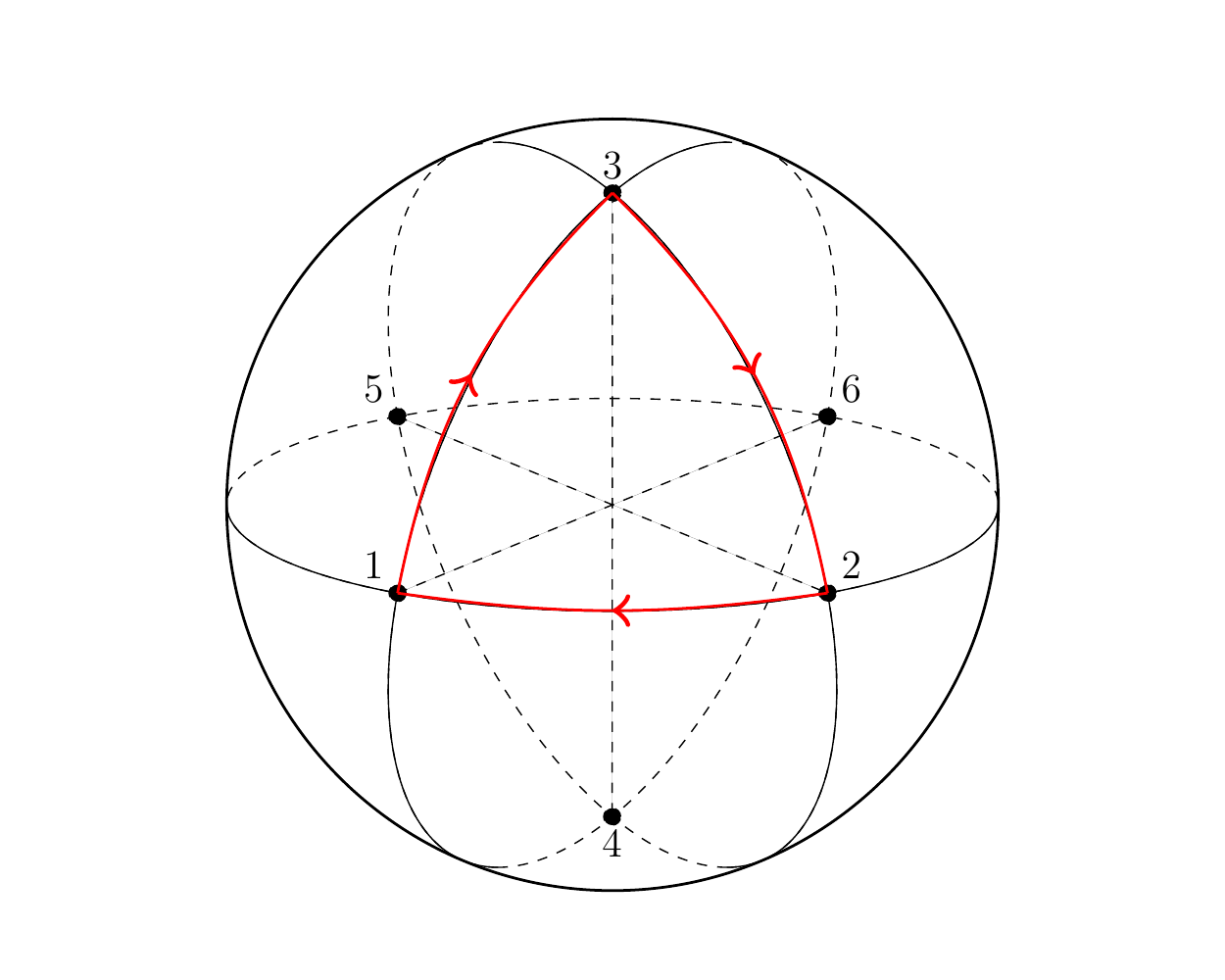}
\caption{Visualization of the triangular closed  path $R_3^{-1}R_1^{-1}R_2^{-1}R_1$ (the path traced by vertex $1$)} \label{triang}
\end{figure}
Very few of these 24 identities are actually independent. First, we notice that one of the generators, e.g. $R_3$ can be expressed as a combination of the other two and their inverses, in several different ways. For example, if we use the first identities in the last two rows we get
$R_3=R_2^{-1}R_1R_2=R_1^{-1}R_2^{-1}R_1$, from which follows
\begin{equation}\label{Artin}
R_2R_1R_2=R_1R_2R_1,
\end{equation}

while if we use $R_3=R_1^{-1}R_2^{-1}R_1=R_1R_2R_1^{-1}$, and also 
$R_3=R_2^{-1}R_1R_2=R_2R_1^{-1}R_2^{-1}$,
we obtain:
\begin{equation}\label{add}
R_1^2=R_2R_1^2R_2,\quad\quad R_2^2=R_1R_2^2R_1.
\end{equation}
All other identities are consequences of these three. Therefore, $G$ is presented as a group generated by two generators and a set of relations, one of which (Equation \ref{Artin}) is precisely Artin's braid relation for the braid group with three strands $B_3$. In other words $G$ is the quotient of $B_3$ by the normal closure of the subgroup generated by the additional relations (\ref{add}).\par
Actually, only one of the identities (\ref{add}) is independent:
\begin{lem}\label{Second}{The second identity in (\ref{add}) follows from the first one and Artin's braid relation (\ref{Artin}).}
\end{lem}
\begin{proof} Using Artin's relation twice it is almost immediate that
$$R_2^{\, 2}=R_1R_2R_1^{\, 2}R_2^{-1}R_1^{-1}.$$
Now using the first of the identities (\ref{add}) inside the expression above we get
$$R_2^{\, 2}=R_1R_2R_1^{\, 2}R_2^{-1}R_1^{-1}=R_1R_2R_2R_1^{\, 2}R_2R_2^{-1}R_1^{-1}=R_1R_2^{\, 2}R_1.$$
\end{proof}
\begin{figure}[h!]
\vskip-10mm
\centering
\includegraphics[width=80mm]{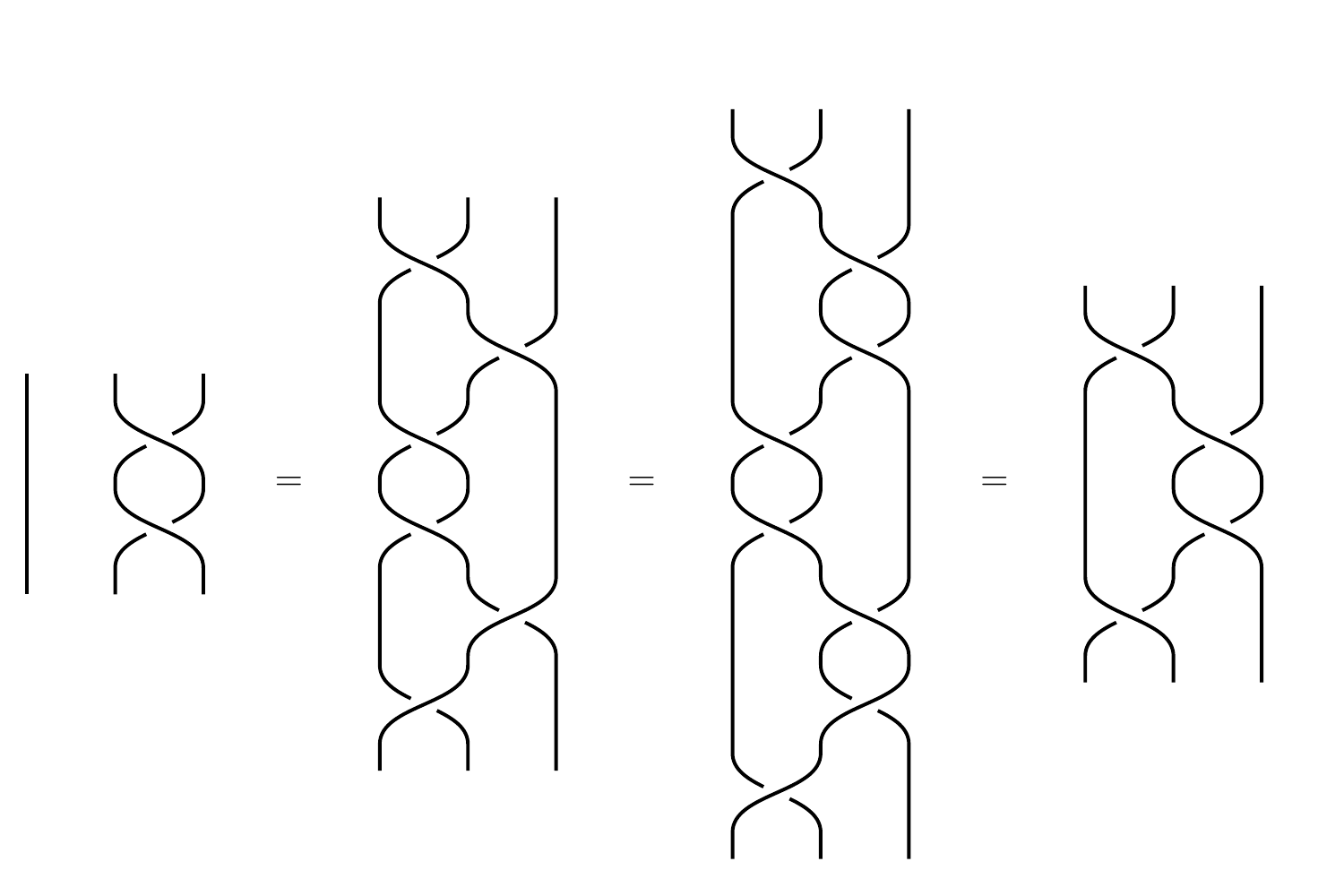}
\caption{Geometric Proof of Lemma \ref{Second} that $R_2^2=R_1R_2^2R_1$ follows from $R_1^2=R_2R_1^2R_2$} \label{braiding}
\end{figure}
\noindent{\bf Note:} As shown by Artin \cite{Artin}, the braid group $B_n$ can be thought as a group of isotopy classes of geometric braids with $n$ strands or as a group generated by $n-1$ generators satisfying what came to be called Artin's braid relations. When $n=3$ there is just one relation (\ref{Artin}) between the two generators. The geometric picture has the advantage of being more intuitive and providing us with ways to see identities, which can then be shown algebraically. Thus, for example, the proof above has a geometric version which is easy to visualize (Figure \ref{braiding}). The same situation will be in place when we consider $n>3$. Our group $G$ will be generated by $n-1$ generators satisfying the standard braid relations plus some additional ones. Using the geometric picture will allow us to arrive at conclusions which are difficult to see algebraically.
\par\smallskip\noindent

{\bf Corollary:} {\it The group $G$ has a presentation}
\begin{equation}\label{Biocta}
G=\left<R_1, R_2\,|\,R_1R_2R_1=R_2R_1R_2, R_1^{\, 2}=R_2R_1^{\, 2}R_2\right>\,.
\end{equation}
\newtheorem{prop}{Proposition}
\begin{prop}\label{R4}
The order of $R_1$ and $R_2$ in $G$ is eight.
\end{prop}
\begin{proof}Using relations (\ref{add}) we conclude that $R_1^{\, 2}=R_2R_1^{\, 2}R_2=R_2^{\, 2}R_1^{\, 2}R_2^{\, 2}$ and similarly $R_2^{\, 2}=R_1^{\, 2}R_2^{\, 2}R_1^{\, 2}$. Therefore
$$R_1^{\, 2}=R_2^{\, 2}R_1^{\, 2}(R_1^{\, 2}R_2^{\, 2}R_1^{\, 2}) \Longrightarrow
Id=R_2^{\, 2}R_1^{\, 4}R_2^{\, 2}\Longrightarrow
R_1^{4}=R_2^{-4}.$$
Next, one can write 
$$R_2^{\, 4}=R_2R_2^{\, 2}R_2=R_2R_1R_2^{\, 2}R_1R_2=R_2R_1R_2R_1R_2R_1=R_1R_2R_1^{\, 2}R_2R_1=R_1^{\, 4}.$$
Putting together the results in the last two equations we see that $R_1^{\, 4}=R_1^{-4}$ and similarly $R_2^{\, 4}=R_2^{-4}$, which imply $R_1^{\, 8}=R_2^{\, 8}=Id$.\par
From this we can tell that the order of $R_1$ and $R_2$ is at most 8. To conclude that it is exactly 8 (and not for example 4) is not a trivial problem. One way to do this is to perform the Todd -- Coxeter algorithm (see e.g. Ken Brown's short description \cite{Brown}). If the group defined by a finite set of generators and relations is finite, the algorithm will (in theory) close and stop and will produce the order of the group and a table for the action of the generators on all elements. We used the simple computer program graciously made available by Ken Brown on his site. The resulting table showed that the order of $G$ is 48 and the central element $R_1^{\, 4}=R_2^{\, 4}$ is not trivial. The results are represented graphically in the so-called Cayley graph (Figure \ref{Cayley}) where dotted lines represent multiplication by $R_1$ and solid lines --- by $R_2$ (The dotted lines wrap around horizontally).
\end{proof}
\begin{figure}[h!]
\centering
\includegraphics[width=100mm]{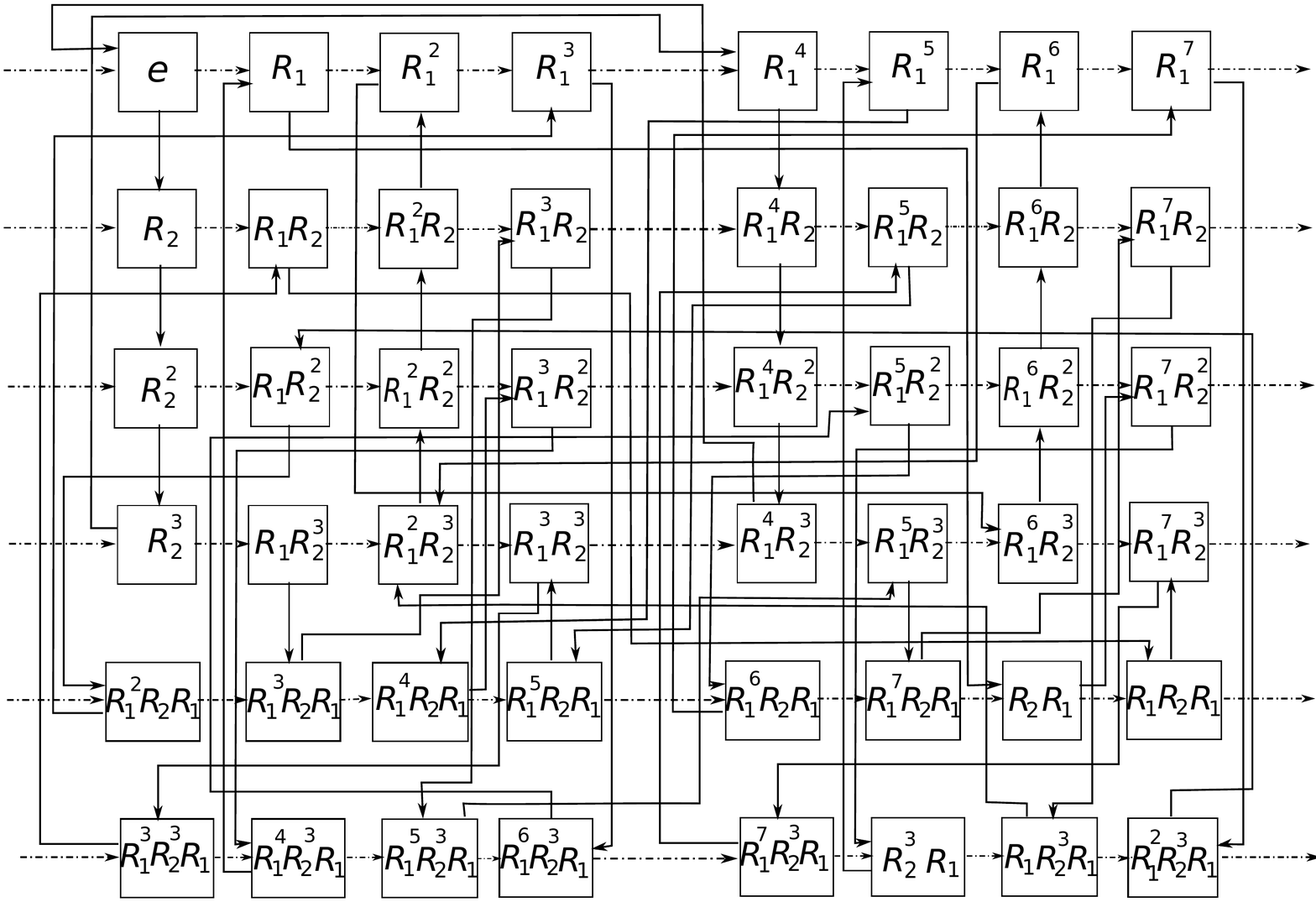}
\caption{The Cayley graph of $G$} \label{Cayley}
\end{figure}
We know that taking the end-point of any element $R\in G$ gives a homomorphism from $G$ onto the group of rotational symmetries of the octahedron and the latter can be considered as a subgroup of the symmetric group $S_6$. So we have a homomorphism $\theta: G \rightarrow S_6$. In particular,
 $\theta(R_1)=(2354)$ and $\theta(R_2)=(1463)$. 
To prove that the fundamental group of $SO(3)$ is $\Zz$, we need to show that the kernel of the homomorphism described above is $\Zz$, because the kernel consists of those elements of $G$ which are classes of closed paths in $SO(3)$. In other words, these are the motions that bring the octahedron back to its original position. It will be helpful to come up with a way to list all elements of $G$ as words in the generators $R_1$, $R_2$ and their inverses, i.e., find a {\it canonical form} for the elements of $G$.
\begin{prop} \label{elements} Any $x\in G$ can be written uniquely in one of the three forms: \\
1. $R_1^{\, m}R_2^{\, n}$ (32 elements), \\
2. $R_1^{\, m}R_2R_1$ (8 elements), \\
3. $R_1^{\, m}R_2^{\, 3}R_1$ (8 elements), \\
where $m\in \{0, 1, 2, 3, . . ., 7\}$ and $n\in \{0, 1 ,2, 3\}$. 
\end{prop}
\begin{proof} Any element of $G$ can be obtained by multiplying the identity by a sequence of the generators and their inverses either on the right or on the left. We are using right multiplication. The idea is to show that 
when multiplying an element $x\in G$, which is written in the form 1, 2 or 3 by any of the two generators of  $G$ or their inverses, we get again an expression of these three types. Since $R_i^{-1}=R_i^7$ it is enough to check the above for positive powers. \par\noindent
The proof is a direct verification using Artin's braid relation and the identity $R_2^{\, 2}R_1=R_1^{-1}R_2^{\, 2}$ (and the symmetric one with $R_1$ and $R_2$ interchanged). Thus, e.g., multiplying the first expression by $R_1$ we get $R_1^{\, m\pm 1}R_2^{\, n}$, if $n$ is even and either expression 2 or 3, if $n$ is odd.\par\noindent
Similarly, multiplying expression 3 by $R_1$ we get

$$
R_1^{\, m}R_2^{\, 3}R_1^{\, 2}=R_1^{\, m+2}R_2^{-3}=R_1^{\, m+2}R_2^{\, 5}=R_1^{\, m+6}R_2
$$
while multiplying it by $R_2$ gives
$$
R_1^{\, m}R_2^{\, 3}R_1R_2=R_1^{\, m}R_2^{\, 2}R_2R_1R_2=R_1^{\, m}R_2^{\, 2}R_1R_2R_1
=R_1^{\, m-1}R_2^{\, 3}R_1\ .
$$
\par\noindent
Uniqueness is proven by inspection. Let us show as an example that $R_1^{\, m}R_2^{\, n}\ne R_1^{\, k}R_2R_1$. Indeed, the assumption that the two are equal leads to the following sequence of equivalent statements:
\begin{eqnarray*}\nonumber
&R_1^{\, m-k}R_2^{\,n}=R_2R_1\Rightarrow   R_1^{\, m-k+1}R_2^{\, n}=R_1R_2R_1=R_2R_1R_2\Rightarrow\\
& R_1^{\, m-k+1}R_2^{\, n-1}=R_2R_1\Rightarrow\dots\Rightarrow
R_1^{\, m-k+n}=R_2R_1\Rightarrow  R_1^{\, m-k+n-1}=R_2.\\
\end{eqnarray*}
The last identity is apparently wrong since no power of $R_1$ can be equal to $R_2$.
\end{proof}
\begin{prop} The kernel of the homomorphism $\theta:G\rightarrow S_6$ is isomorphic to $\Zz$. 
\end{prop}
\begin{proof} We study how the images under $\theta$ of the elements of $G$ permute the numbers $\{1,2,3,4,5,6\}$. First, $\theta(R_1)$ leaves $1$ in place, then $\theta(R_2)$ sends $1$ to $4$, while  $\theta(R_2^{\, 3})$ sends $1$ to $3$ . Finally, recalling that $\theta(R_1)=(2354)$ we see that $1$ is not in the orbit of $4$ or $3$ under the action of  $\theta(R_1)$, so there is no way that $\theta(R_1^{\, m}R_2R_1)$ or $\theta(R_1^{\, m}R_2^{\, 3}R_1)$ can bring $1$ back to itself. Next we look at  $\theta(R_1^{\, m}R_2^{\, n})$. Recall that because $R_2^{\, 4}=R_1^{\, 4}$ and $R_1^{\, 8}=Id$ we have $n\in\{0,1,2 ,3\}$ and $m\in \{0,\dots,7\}$.Unless $\theta(R_2^{\, n})=Id$, which happens if and only if $n=0$, one of the numbers $\{3,4,6\}$ will be sent to $1$. Then, since $1$ is fixed by $\theta(R_1^{\, m})$, we see that  $\theta(R_1^{\, m}R_2^{\, n})$ cannot be trivial unless $n= 0$. In that case it is obvious that $\theta(R_1^{\, m})$ can be the identity only when $m=0\,(\text{mod}\,4)$.
Thus, 
$$
\pi_1(SO(3))\cong \text{ker}(\theta)=\{Id, R_1^{\, 4}\}\cong \Zz.
$$
\end{proof}
Since $R_1^{\, 4}$ is a central element of order two, the group generated by it, $\{Id, R_1^{\, 4}\}$, belongs to the center $Z(G)$. It is not obvious a priori that there are no other central elements. We make the following observation: the center of the braid group $B_n$ is known to be isomorphic to $\Z$. It is generated by a full twist of all $n$ strands --- $(\sigma_1\sigma_2\cdots\sigma_{n-1})^n$ (using $\sigma_i$ for the generators of $B_n$). Since $G$ is obtained from $B_n$ by imposing one more relation, any central element in $B_n$ will be central in $G$, although there is no guarantee that we will obtain a nontrivial element in $G$ in this way. In addition, there may be central elements in $G$ which come from non-central elements in $B_n$. When $n=3$ the above argument ensures that $(R_1R_2)^3$ is a central element in $G$. The calculation in Proposition \ref{R4} shows that 
$$
(R_1R_2)^3=R_1R_2R_1R_2R_1R_2=R_1^{\, 4}=R_2R_1R_2R_1R_2R_1=R_2^{\, 4},
$$
so we do not get any new central element, different from the one we have already found. When $n=4$, however the analogous calculation, using the braid relations plus the additional relations as in equation \ref{add}, gives
$$
(R_1R_2R_3)^{\, 4}=R_1^{\, 2}R_3^{-2}.
$$
This is another central element, different from $R_1^{\, 4}=R_2^{\, 4}=R_3^{\, 4}$. Taking the product of the two we obtain a third central element $R_1^{\, 2}R_3^{\, 2}$. Therefore when $n=4$, $Z(G)$ contains (in fact coincides with) the product of two copies of $\Zz$. This obviously generalizes to any even $n$ --- the element $R_1^{\, 2}R_3^{\, 2}\cdots R_{n-1}^{\, 2}$ is central, as can be checked explicitly.\par
The difference between even and odd dimensions can be traced back to the difference between rotation groups in even and odd dimensions. As we shall see shortly, when we factor $G$ by $\text{ker}\,\theta$ we obtain the rotational hyperoctahedral group. In odd dimensions this has trivial center, which follows for example from Schur's lemma and the fact that the matrix $-\mathbf{1}$ is not a rotation. However, in even dimensions reflection of all axes, given by $-\bold{1}$, is a rotation and therefore the corresponding rotational hyperoctahedral group has a nontrivial center.\par
According to the general construction, we expect that when we factorize $G$ by the kernel of the covering map, which is nothing but $\text{ker}\, \theta$, we should obtain the rotational octahedral group. This can also be established directly. Denoting by $G_1$ the quotient, we have a presentation for it:
\begin{equation}\label{S4}
G_1=\left<R_1, R_2 \,|\, R_1R_2R_1=R_2R_1R_2, R_1^{\, 2}=R_2R_1^{\, 2}R_2, R_1^{\, 4}=Id\right>.
\end{equation}
The order of $G_1$ is 24. We can list its elements using the same expressions as in Proposition \ref{elements}, except that now the two integers $n$ and $m$ run from $0$ to $3$. An easy calculation shows that $G_1$ contains 
nine elements of order 2, eight elements of order 3, and six elements of order 4. Among all 15 classified groups of order 24, the only one that has this structure is the symmetric group $S_4$, which on the other hand is the group of the octahedral rotational symmetries.
 Given the structure of $G$ that can be visualized through its Cayley graph, we conclude that it is a nontrivial extension by $\Zz$ of $S_4$ described by a non-split short exact sequence
$$
1\longrightarrow \Zz\longrightarrow G\longrightarrow S_4\longrightarrow 1\ .
$$
The non-isomorphic central extensions by $\Zz$ of $S_4$ are in one-to-one correspondence with the elements of the second cohomology group (for trivial group action) of $S_4$ with coefficients in $\Zz$ and the latter is known to be isomorphic to the Klein four-group $\Zz\times \Zz$. The identity in cohomology corresponds to the trivial extension $\Zz\times S_4$ (more generally, a semidirect product is also considered trivial). The other three elements of the cohomology group classify the three non-isomorphic nontrivial extensions, namely the binary octahedral group $2O$, the group $GL(2,3)$
of nonsingular $2\times 2$ matrices over the field with three elements, and the group $SL(2,4)$ of $2\times 2$ matrices with unit determinant over the ring of integers modulo 4.\par\noindent 
\begin{prop}\label{2O}
$G$ is isomorphic to the binary octahedral  group $2O$.
\end{prop}
\noindent{\bf Note:}  By construction, we expect the group $G$ to be a subgroup of $SU(2)$ and the covering map $G\rightarrow S_4$ to be a restriction of the covering map $SU(2)\rightarrow SO(3)$. The fact that $G$ with the presentation (\ref{Biocta}) is isomorphic to $2O$ is mentioned in Section 6.5 of \cite{Cox}. The GAP ID of $G$ is [48, 28].
\begin{proof}
The binary octahedral group, being a subgroup of $SU(2)$, can be realized as a group of unit quaternions. In fact it consists of the 24 Hurwitz units 
$$\{\pm1, \pm i, \pm j, \pm k, \frac{1}{2}(\pm 1 \pm i \pm j \pm k)\}$$ and the following 24 additional elements: 
$$\{\frac{1}{\sqrt{2}}(\pm 1 \pm i),\frac{1}{\sqrt{2}}(\pm 1 \pm j),\frac{1}{\sqrt{2}}(\pm 1\pm k), \frac{1}{\sqrt{2}}(\pm i \pm j),\frac{1}{\sqrt{2}}(\pm i\pm k),
\frac{1}{\sqrt{2}}(\pm j\pm k)\}$$ 
To prove the isomorphism it is enough to find two elements in $2O$ which generate the whole group and which satisfy the same relations as the defining relations of $G$. Obviously we are looking for elements of order 8.
Let, e.g., $u_1=\frac{1}{\sqrt{2}}(1-k), u_2=\frac{1}{\sqrt{2}}(1-j)\in 2O$. It is a simple exercise in quaternion algebra to prove that  $u_2u_1u_2=u_1u_2u_1$. Similarly, we check that $u_1^{\, 2}=u_2u_1^{\, 2}u_2$, or equivalently that $u_1^{\, 2}u_2=u_2^{-1}u_1^{\, 2}$. Indeed, since $u_1^{\, 2}=-k$ and since $k$ and $j$ anticommute,
$$
u_1^{\, 2}u_2=-k\frac{1}{\sqrt{2}}(1-j)=\frac{1}{\sqrt{2}}(1+j)(-k)=u_2^{-1}u_1^{\, 2} \ .
$$
A direct verification further shows that the whole $2O$ is generated by $u_1$ and $u_2$.
\end{proof}
The 24 Hurwitz units, when considered as points in $\bBR ^4$, lie on the unit sphere $S^3$ and are the vertices of a regular 4-polytope   --- the 24-cell, one of the exceptional regular polytopes with symmetry --- the Coxeter group $F_4$. The set is also a subgroup of $2O$ --- the binary tetrahedral group, denoted as $2T$. The second set of 24 unit quaternions can be thought of as the vertices of a second 24-cell, obtained from the first one  by a rotation,  given by multiplication of all Hurwitz units by a fixed element, e.g. ${1\over \sqrt{2}}(1+i)$. The convex hull of all 48 vertices is a  4-polytope, called disphenoidal 288-cell.\par
It may be instructive to consider the symmetric group $S_4$ with its presentation given by equation (\ref{S4}) and try to construct explicitly all non-isomorphic central extensions by $\Zz$. This means that the three relations in the presentation of $S_4$ must now be satisfied up to a central element, belonging to the (multiplicative) cyclic group with two elements $\{1, -1\}$:
$$
R_1R_2R_1=aR_2R_1R_2,\quad R_1^{\, 2}=bR_2R_1^{\, 2}R_2,\quad R_1^{\, 4}=c,\quad a,b,c\in \Zz.
$$
The element $a$ in the first relation can be absorbed by replacing $R_1$ by $aR_1$, so Artin's braid relation remains unchanged. We are left with four choices for $b$ and $c$:\\
1. The choice $b=c=1$ leads to the trivial extension as a direct product $\Zz\times S_4$.\\ \smallskip
2. The choice $b=1, c=-1$ leads to the already familiar group 
$$
G\cong 2O=\left<R_1, R_2\,|\,R_1R_2R_1=R_2R_1R_2, R_1^{\, 2}=R_2R_1^{\, 2}R_2\right>.
$$
3. The choice $b=c=-1$ leads to the group  $GL(2,3)$ with presentation
$$
GL(2,3)=\left<R_1, R_2\,|\,R_1R_2R_1=R_2R_1R_2, R_1^{\, 2}=R_2R_1^{\, 6}R_2, R_2R_1^{\, 4}=R_1^{\, 4}R_2\right>.
$$
4. The choice $b=-1, c=1$ leads to the group $SL(2,4)$ with presentation 
$$
SL(2,4)=\left<R_1, R_2,b\,|\,R_1R_2R_1=R_2R_1R_2, R_1^{\, 2}=bR_2R_1^{\, 2}R_2,R_1^{\, 4}=b^2=1\right>.
$$
or, after simplifications, to
$$
SL(2,4)=\left<R_1, R_2\,|\,R_1R_2R_1=R_2R_1R_2, R_1^{\, 4}=(R_1R_2)^6=1\right>.
$$
The proofs of points 3 and 4 above repeat the logic of the proof of Proposition \ref{2O}. For the group $GL(2,3)$ we can make the following identifications
$$
R_1=\begin{pmatrix}
1&1\\
1&0
\end{pmatrix},
R_2=\begin{pmatrix}
1&2\\
2&0
\end{pmatrix}\in GL(2,3)
$$
and then check that they generate the whole $GL(2,3)$ and satisfy the respective relations. Similarly, for $SL(2,4)$ we can set
$$
R_1=\begin{pmatrix}
1&0\\
1&1
\end{pmatrix},
R_2=\begin{pmatrix}
3&3\\
0&3
\end{pmatrix}
\in SL(2,4).
$$
Notice that there is an essential difference between the extensions described in 2 and 3, and the extension in 4. The extensions in 2 and 3 have presentations as the presentation of $S_4$ (equation \ref{S4}) with the same set of generators and some relation removed. This is not the case with $SL(2,4)$ where, if we want to keep the form of the relations, we need three generators. Notice also that for $2O$ and $GL(2,3)$ the order of $R_1$ and $R_2$ becomes 8 (it is 4 in $S_4$), while in $SL(2,4)$ the order of $R_1$ and $R_2$ remains 4. (It may be worth pointing out that the properties of the braid group $B_n$ ensure that in any of its factors the order of any two (standard) generators has to be the same.) \par
\begin{figure}[h!]
\centering
\includegraphics[width=50mm]{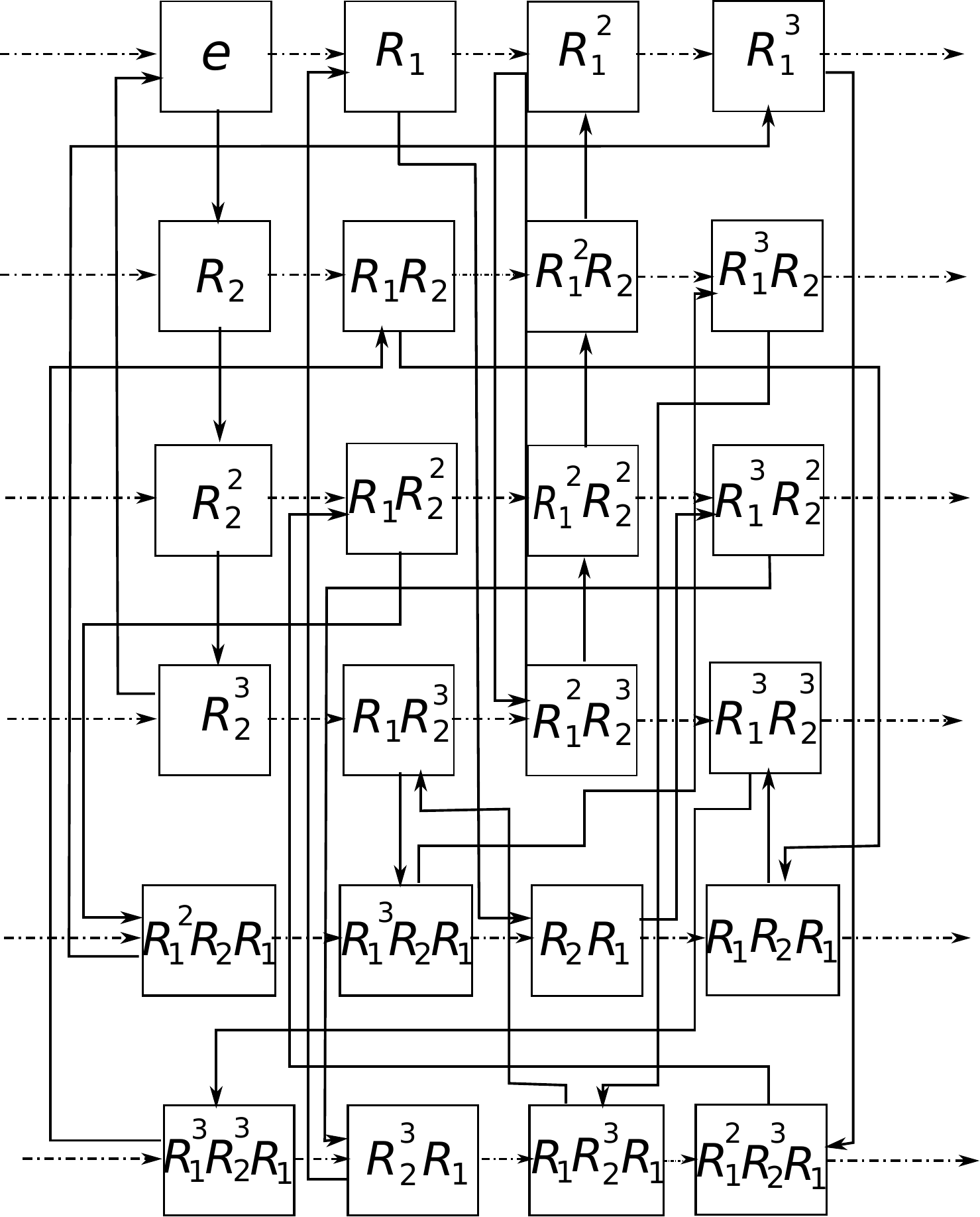}
\caption{The Cayley graph of $S_4$} \label{Cayley24}
\end{figure}
The groups $2O$ and $GL(2,3)$ are so-called {\it stem extensions} of $S_4$, i.e., the abelian group by which we extend is not only contained in the center of the extended group, but also in its commutator subgroup. One can check explicitly that for $2O$ and $GL(2,3)$ the corresponding central element $R_1^{\, 4}=R_2^{\, 4}$ can be written as a commutator, while in the case $SL(2,4)$ the center is generated by $(R_1R_2)^3$ and it is not in the commutator subgroup. If one looks at the Cayley graph of $2O$ (Figure \ref{Cayley}), one can see that it really looks like a (topological) double cover of the corresponding Cayley graph of $S_4$ (Figure \ref{Cayley24} ). For this reason stem extensions are called covering groups and are discrete versions of covering groups of Lie groups. It is intriguing that while $SO(3)$ has $SU(2)$ as its unique (double) cover, the finite subgroup $S_4$ has one additional double cover, namely $GL(2,3)$, which does not come from lifting $SO(3)$ to $SU(2)$.\par
Finally, we may notice that in the presentations of $2O$ and $GL(2,3)$ there is no condition imposed on the order of the central element. Its order comes out to be 2 automatically, which means that $S_4$ does not admit bigger stem extensions and the two groups considered are maximal stem extensions, i.e., {\it Schur extensions}.
\section{Generalization to arbitrary $\boldsymbol{n}$}
The $n$-dimensional case is an easy generalization of the three-dimensional one. We consider products of generating paths $R_{ij}$ in $SO(n)$ as defined in Section 2. If we take a closed path of the form
$R_{ij}R_{kl}R_{ij}^{-1}R_{kl}^{-1}$, where all four indices are different, it is clear that this is contractible as these are rotations in two separate planes and at the level of homotopy classes the generators $R_{ij}$ and $R_{kl}$ will commute. When one of the indices coincides, we have a motion that takes place in a 3-dimensional subspace of $\Rn$ and we may use the algebraic relations we had in the previous section. In particular, if we consider the elements $R_{ij}$, $R_{jk}$ and $R_{ki}$, we can express one, e.g. the third, in terms of the other two, just as we expressed $R_3=R_{12}$ as the conjugation of $R_2=R_{31}$ by $R_1=R_{23}$. At this point, it seems convenient to choose a different notation, where $R_1:=R_{12}, 
R_2:=R_{23},\dots, R_{n-1}:=R_{n-1\,n}$. With this notation we have $R_{31}=R_{13}^{-1}=R_1R_2R_1^{-1}$, 
$R_{41}=R_{13}R_{34}R_{13}^{-1}=R_1R_2^{-1}R_1^{-1}R_3R_1R_2R_1^{-1}$, etc. In this way, all $R_{ij}$ are products of the $n-1$ generators $R_i$ and their inverses. We have $R_iR_j=R_jR_i$ when $|i-j|\ge 2$. Further, since local closed paths that correspond to rotations in any 3-dimensional subspace must be set to identity when passing to homotopy classes, we have identities analogous to the ones in equation \ref{Id} for any two generators $R_i$ and $R_{i+1}$. In particular, Artin's braid relation is satisfied for any $i=1,\dots , n-2$:
$$
R_iR_{i+1}R_i=R_{i+1}R_iR_{i+1}.
$$
The other relation must also hold:
$$
R_iR_{i+1}^{\, 2}R_i=R_{i+1}^{\, 2}.
$$
The fact that there are no additional relations follows from the observation that any contractible closed path in $SO(n)$ which is a product of generating paths can be written as a product of "triangular" local closed paths of the type described in Section 3 (see Appendix).\par
The properties of the braid group lead to some interesting restrictions on the type of additional relations that can be imposed on the generators. In particular, if $P(R_1,R_2)=1$ is some relation involving the first two generators and their inverses, then it follows that $P(R_i,R_{i+1})=1$ (translation) and $P(R_{i+1},R_i)=1$ (symmetry) will be satisfied automatically. These have a simple geometric explanation. For example,  the first relation involves the first three strands of the braid, so if we want to prove the relation for $R_2$ and $R_3$, we flip the first strand over the next three, apply the relation and then flip back the former first strand to the first place. The symmetry property follows from the first and the fact that the braid group has an outer automorphism $R_i\rightarrow R_{n-i}$ which correspond to "looking at the same braid from behind." It is also clear that the order of all generators in the factor group will be the same. Therefore, for any $n$ the group $G$ generated by the  generating paths, up to homotopy, has presentation
\begin{small}
\begin{equation}\label{n-dim}
 G=\left< R_1,\dots,R_{n-1}\,|R_iR_{i+1}R_i=R_{i+1}R_iR_{i+1}; \\
 R_iR_j=R_jR_i, |i-j|\ge 2; R_1^{\, 2}=
R_2R_1^{\, 2}R_2\right>.
\end{equation}
\end{small}
The group $G$ for arbitrary $n$ has many features in common with the case $n=3$. In particular, all $R_i$ have order $8$ and $R_1^{\, 4}=R_2^{\, 4}=\cdots=R_{n-1}^{\, 4}$ is central.
The Todd-Coxeter algorithm when run on the computer gives the following results for the order of $G$ when $n=3, 4, 5, 6$ respectively --- 48, 384, 3840, 46080. In fact, we have $|G|=2^nn!$, which will be shown next. As mentioned already in Section 3, the full hyperoctahedral group in dimension $n$, which is the Coxeter group $B_n$, has the same order, but $G$ is a different group --- it is a non-trivial double cover of the orientation-preserving subgroup of the full hyperoctahedral group.\par
A canonical form of the elements of $G$ can be defined inductively as follows:\\
Let $x_{(i)}\in G$ denote a word which contains no $R_j$ with $j>i$. Then we will say that $x_{(i)}$ is in canonical form if it is written as $x_{(i)}=x_{(i-1)}y_{(i)}$ with $x_{(i-1)}$ being in canonical form and $y_{(i)}$ being an expression of one of the types:\\
1. $R_i^{\, k},\ k\in\{0,1,2,3\}$,\\
2. $R_iR_{i-1}\cdots R_{i-j},\ j\in\{1,\dots,i-1\}$,\\
3.  $R_i^{\, 3}R_{i-1}\cdots R_{i-j},\ j\in\{0,\dots,i-1\}$,\\
As an example, let us list all elements of $G$ in the case $n=4$, by multiplying all canonical expressions containing $R_1$ and $R_2$ (see Proposition \ref{elements}) with all the expressions as above, with $i=3$. We have $R_1^{\, m}R_2^{\, n}R_3^{\, k}$ (128 elements), $R_1^{\, m}R_2^{\, n}R_3R_2$ (32 elements), $R_1^{\, m}R_2^{\, n}R_3R_2R_1$ (32 elements), $R_1^{\, m}R_2^{\, n}R_3^{\, 3}R_2$ (32 elements), $R_1^{\, m}R_2^{\, n}R_3^{\, 3}R_2R_1$ (32 elements),  $R_1^{\, m}R_2R_1R_3^{\, k}$ (32 elements),  $R_1^{\, m}R_2R_1R_3R_2$ (8 elements),
 $R_1^{\, m}R_2R_1R_3R_2R_1$ (8 elements),  $R_1^{\, m}R_2R_1R_3^{\, 3}R_2$ (8 elements), 
$R_1^{\, m}R_2R_1R_3^{\, 3}R_2R_1$ (8 elements), \\
$R_1^{\, m}R_2^{\, 3}R_1R_3^{\, k}$ (32 elements),  $R_1^{\, m}R_2^{\, 3}R_1R_3R_2$ (8 elements),
 $R_1^{\, m}R_2^{\, 3}R_1R_3R_2R_1$ (8 elements),  $R_1^{\, m}R_2^{\, 3}R_1R_3^{\, 3}R_2$ (8 elements),
$R_1^{\, m}R_2^{\, 3}R_1R_3^{\, 3}R_2R_1$ (8 elements).
The following is a straightforward generalization of Proposition \ref{elements}.
\begin{prop}
Any element of $G$ can be written uniquely in the canonical form defined above. The number of elements is  $2^nn!$.
\end{prop}
\begin{proof}
The idea is to show that by multiplying an element in canonical form on the right by any $R_j$, one gets another element that can be brought to a canonical form as well. There is nothing conceptually different from the proof of Proposition \ref{elements} and we skip the details. In order to calculate the order of $G$, we notice that the elements in canonical form of type $x_{(i)}$ are obtained by all possible products of all elements of type $x_{(i-1)}$ with the $4+2(i-1) =2(i+1)$ different expressions of type $y_{(i)}$. Starting with $i=2$ where we have $48=2^33!$ and remembering that $i$ runs from $1$ to $n-1$ we get 
$$
|G|=2^3\cdot 3!\cdot 2\cdot 4\cdot 2 \cdot 5\cdots 2\cdot n=2^n n!.
$$
\end{proof}
The group $G$ consists of homotopy classes of paths in $SO(n)$ starting at the identity and ending at an element of $SO(n)$ which is a rotational symmetry of the hyperoctahedron in $n$ dimensions (also called $n$-orthoplex or $n$-cross polytope). In particular we have a homomorphism $\theta :G\rightarrow S_{2n}$ because we have a permutation of the $2n$ vertices of the hyperoctahedron (we take left action of the group on the set). Since the fundamental group of $SO(n)$ consists of the homotopy classes of closed paths, we investigate the kernel of $\theta$. 
\newtheorem*{theo}{Theorem}
\begin{theo}
$$
\pi_1(SO(n))\cong \text{ker}\theta=\{Id, R_1^{\, 4}\}\cong \Zz.
$$
\end{theo}
\begin{proof}
We choose to enumerate the $2n$ vertices of the hyperoctahedron as $\{1,-1,2,$ \\ $-2,\dots n,-n\}$ taking $\pm 1$ to denote the two opposite vertices on the first axis, $\pm 2$ on the second axis, etc. Then the element $\theta(R_i)$ permutes cyclically only the elements $\{i, i+1,-i,-(i+1)\}$ leaving the rest in place. Consider an element  $x_{(n-1)}=x_{(n-2)}y_{(n-1)}\in G$ in canonical form. By looking at the three possible expressions for $y_{(n-1)}$ we see that the leftmost letter is $R_{n-1}$ to some power, possibly preceded (on the right) by a sequence of $R_i$s with decreasing $i$. Since in the word $y_{(n-1)}$ only $\theta(R_{n-1})$ moves $n$ and $-n$, we have $\theta(y_{(n-1)})(n)\ne n$, unless $y_{(n-1)}=1$. Now, because $x_{(n-2)}$ does not contain $R_{n-1}$, the number $n$ is not  the image of any number $k\ne n$ under the action of $\theta(x_{(n-2)})$. Therefore we have $\theta(x_{(n-1)})(n)\ne n$, unless $x_{(n-1)}=x_{(n-2)}$. Proceeding in this way we see that $x\in ker \,\theta$ if and only if $x=R_1^{\, m}$. Finally, as $\theta(R_1)=(1\ 2-1-2)$ (cyclic permutation of $\{1,2,-1,-2\}$; the rest fixed), the only possible cases are when $m=0, 4$. Therefore, $R_1^{\, 4}$ is the only nontrivial element of $ker\,\theta$ and it has order $2$.
\end{proof}
We obtained a series of finite groups $G$ from the braid groups $B_n$ by imposing one additional relation, namely $R_1^{\, 2}=R_2R_1^{\, 2}R_2$. These groups are nontrivial double covers of the corresponding rotational hyperoctahedral groups and have order $2^nn!$. It is quite obvious that we obtain a second, nonisomorphic series of double covers if we impose the relations $R_1^{\, 2}=
R_2R_1^{\, 6}R_2$ and $R_1^{\, 4}R_2=R_2R_1^{\, 4}$, instead. Note that for the second series we need the additional condition, which then implies that $R_1^{\, 4}=R_2^{\, 4}=\cdots=R_{n-1}^{\, 4}$ is central. When $n=3$ the two groups are the two Schur extensions of the base group. This is perhaps the case also for arbitrary $n$.
\section{Appendix}

{\it Proof of Lemmas 1 and 2.} It is helpful to introduce a function, measuring the (square of a) "distance" between two points on $SO(n)$. Let $X,\ Y\in SO(n)$ be written as $n\times n$ orthogonal matrices and define
$$
D(X,Y)\defeq \Tr({\mathbf 1}- X^TY)\ .
$$
It is a simple exercise to show that this is a positive-definite symmetric function on $SO(n)\times SO(n)$.
Notice that $(X^TY)_{ii}$ is the cosine of the angle between the image of the 
standard coordinate basis vector $\mathbf e_i$ under the action of $X$ and its image under the action of $Y$. The function $D$ does not satisfy the triangle inequality and is not a true distance, but this causes no difficulties in our considerations.\par
Further we write the proof for $SO(3)$ for brevity. It is obvious that the same method works in general. Let $X\in SO(3)$ be written as a $3\times 3$ orthogonal matrix:
$$ 
X=\begin{pmatrix}
x_{11}&x_{12}&x_{13}\\
x_{21}&x_{22}&x_{23}\\
x_{31}&x_{32}&x_{33}
\end{pmatrix}
\equiv\begin{pmatrix}
{\mathbf x_1}&{\mathbf x_2}&{\mathbf x_3}
\end{pmatrix}.
$$
The "distance" of $X$ to the identity (writing just one argument) is :
$$
D(X)=3-x_{11}-x_{22}-x_{33}.
$$

Thus $D(X)\ge 0$ and $D(X)=0$ implies $X=\mathbf 1$. The gradient of $D(X)$ is a 9-dimensional vector field, which we choose to write as a $3\times 3$ matrix. We have
$${\rm grad}\, D(X)=
\begin{pmatrix}
{\mathbf e_1}&{\mathbf e_2}&{\mathbf e_3}
\end{pmatrix}.
$$
$SO(3)$ is a 3-dimensional submanifold of $\bBR^9$ consisting of the points satisfying six algebraic equations, ensuring that the $3\times 3$ matrix $X$ is orthogonal. In vector form these are equivalent to the statement that the (column) vectors $\{{\mathbf x_1},{\mathbf x_2},{\mathbf x_3}\}$ form an orthonormal basis:
\begin{small}
$$
{\mathbf x_1}\cdot {\mathbf x_1}=1,\quad {\mathbf x_2}\cdot {\mathbf x_2}=1,\quad 
{\mathbf x_3}\cdot {\mathbf x_3}=1,\quad {\mathbf x_1}\cdot {\mathbf x_2}=0,\quad 
{\mathbf x_1}\cdot {\mathbf x_3}=0,\quad {\mathbf x_2}\cdot  {\mathbf x_3}=0.
$$ 
\end{small}
The respective gradients of these six functions are
$$
\begin{pmatrix}
{\mathbf x_1}&{\mathbf 0}&{\mathbf 0}
\end{pmatrix}, \quad
\begin{pmatrix}
{\mathbf 0}&{\mathbf x_2}&{\mathbf 0}
\end{pmatrix}, \quad
\begin{pmatrix}
{\mathbf 0}&{\mathbf 0}&{\mathbf x_3}
\end{pmatrix}, \quad
$$
$$
\begin{pmatrix}
{\mathbf x_2}&{\mathbf x_1}&{\mathbf 0}
\end{pmatrix},\quad
\begin{pmatrix}
{\mathbf x_3}&{\mathbf 0}&{\mathbf x_1}
\end{pmatrix},\quad
\begin{pmatrix}
{\mathbf 0}&{\mathbf x_3}&{\mathbf x_2}
\end{pmatrix}.
$$
It is well known and an easy exercise that for each $X\in SO(3)$ these six vectors are linearly independent and their orthogonal complement is precisely the tangent space of $SO(3)$. (This is in fact how it is shown that these six  equations in $\bBR^9$ define indeed a three-dimensional submanifold.) We want to show that under our assumptions ${\rm grad}\, D(X)$ has a nonzero tangential component. Indeed, supposing that ${\rm grad}\, D(X)$ is in the span of the six vectors above leads to
\begin{eqnarray}\nonumber
\begin{pmatrix}
{\mathbf e_1}&{\mathbf e_2}&{\mathbf e_3}
\end{pmatrix}&=&
a\begin{pmatrix}
{\mathbf x_1}&{\mathbf 0}&{\mathbf 0}
\end{pmatrix}+
b\begin{pmatrix}
{\mathbf 0}&{\mathbf x_2}&{\mathbf 0}
\end{pmatrix}+
c\begin{pmatrix}
{\mathbf 0}&{\mathbf 0}&{\mathbf x_3}
\end{pmatrix}\\\nonumber
&+&
d\begin{pmatrix}
{\mathbf x_2}&{\mathbf x_1}&{\mathbf 0}
\end{pmatrix}
+
e\begin{pmatrix}
{\mathbf x_3}&{\mathbf 0}&{\mathbf x_1}
\end{pmatrix}
+
f\begin{pmatrix}
{\mathbf 0}&{\mathbf x_3}&{\mathbf x_2}
\end{pmatrix},\\\nonumber
\end{eqnarray}
which in turn is equivalent to the three vector equations
$$
{\mathbf e_1}=a{\mathbf x_1}+d{\mathbf x_2}+e{\mathbf x_3},\quad
{\mathbf e_2}=b{\mathbf x_2}+d{\mathbf x_1}+f{\mathbf x_3},\quad
{\mathbf e_3}=c{\mathbf x_3}+e{\mathbf x_1}+f{\mathbf x_2}.
$$
Taking the scalar product of the first equation with ${\mathbf x_2}$ yields $d=x_{12}$, while taking the scalar product of the second equation with ${\mathbf x_1}$ yields $d=x_{21}$. In a similar way we see that $x_{ij}=x_{ji}$ for any $i,j$. Therefore the matrix $X$ must be symmetric and being also orthogonal its square is the identity. The eigenvalues can only be $1$ and $-1$ but the latter is excluded by the assumption that under the transformation corresponding to $X$ the coordinate axes do not leave the closed half-space they belong to initially.\par
Thus that the tangential component of $-{\rm grad}\, D(X)$ defines a vector field on $SO(3)$ which will be nonzero for any $X=R(t)$ where $R:[0,1]\rightarrow SO(3)$ is a local closed path. This means that the flow along this vector field defines a homotopy from  $R(t)$ to the identity of $SO(3)$. This proves the first part of Lemma 2.
\par
Let us denote by $G'\subset SO(n)$ the respective rotational hyperoctahedral group in dimension $n$.
Taking an arbitrary path $R: [0,1]\rightarrow SO(n)$ with $R(0)=Id$ and $R(1)=r\in G'$, we want to construct another path, homotopic to the first one, which is a product of generating paths $R_i$. We can proceed as follows: If $t_1\in [0,1]$ the smallest $t$ for which the "distance" from $R(t)$ to some $r_1\in G'$ becomes equal to the "distance" to $Id$, we take a product of generating paths $R_{k_1}^{(1)}\cdots R_1^{(1)}$ with $(R_{k_1}^{(1)}\cdots R_1^{(1)})(0)=Id$ and $(R_{k_1}^{(1)}\cdots R_1^{(1)})(t_1)=r_1$. Note that the transformation  $(R_{k_1}^{(1)}\cdots R_1^{(1)})(t)$ leaves every vertex of the hyperoctahedron in the closed half-space determined by it (as defined in Section 2) or, equivalently, the angle between the standard basis vector ${\mathbf e_i}$ and its image under 
$(R_{k_1}^{(1)}\cdots R_1^{(1)})(t)$ does not exceed $\pi/2$ for any $i$ and any $t$. Indeed, as we consider a continuous path in $SO(n)$ starting at the identity, we may think of the motion of the $n$ points on the $(n-1)$--dimensional sphere (the images of the vectors ${\mathbf e_i}$ under $R(t)$). Initially all points remain in some respective adjacent $(n-1)$-cells (these are $(n-1)$--simplices (spherical)) forming the spherical hyperoctahedron. If the point $i$ is to leave the closed half-space to which it initially belonged, it must reach the $(n-2)$-dimensional boundary opposite to ${\mathbf e_i}$ for some $t$. (At the same time there will be at least one more point $i'$ belonging to a boundary of an $(n-1)$-cell since two points cannot belong to the interior of the same $(n-1)$--cell).  But then it follows that there will be a vertex $j$ the angular distance  to which, from $i$  is less than or equal to $\pi/2$. This implies that the element $r_{ij}\in G'$ giving rotation by $\pi/2$ in the $ij$th plane is not further to $R(t)$ than the "distance" between $R(t)$ and the identity. The above argument shows that even though the element $r_1$ above does not determine the product $R_{k_1}^{(1)}\cdots R_1^{(1)}$ uniquely, it is unique up to homotopy.\par
Next, we take $t_2$ as the smallest $t\ge t_1$ for which the "distance" from $R(t)$ to some $r_2$ becomes equal to the "distance" to $r_1$. There is an element $r_2'$ in $G'$, such that $r_2=r_2'r_1$ and $r_2'$ satisfies the same property as $r_1$ above. We take a product of generating paths $R_{k_2}^{(2)}\cdots R_1^{(2)}$ with $(R_{k_2}^{(2)}\cdots R_1^{(2)})(t_1)=Id$ and $(R_{k_2}^{(2)}\cdots R_1^{(2)})(t_2)=r_2'$. Proceeding in this way and taking the product of products of generating paths we produce a path (after renumbering) $R'\defeq R_k\cdots R_1$ with $(R_k\cdots R_1)(0)=Id$ and $(R_k\cdots R_1)(1)=r$. 
The elements $r_i$ which we have to use at each step may not be unique and we will have to make a choice but the end result will lead to homotopic paths.
To show that the path we have constructed is homotopic to the original path $R$ we can use the following trick --- for each fixed $t$ take $R''(t)\defeq R'(t)R^{-1}(t)$. The path $R''$ is closed and is local by the construction of $R'$. Therefore it is homotopic to the identity by the first part of Lemma 2. With this Lemma 1 is proven.\par
Finally, we need to show that the expression corresponding to any contractible closed path consisting of generating paths  can be reduced to the identity by inserting in it words giving local closed paths. More precisely, the words that must be inserted give triangular closed paths as described in Section 3. This is essential as we must be sure that there are no additional relations other than the ones as in Equations \ref{Artin} and \ref{add}.
First we carry out the proof for local closed paths consisting of generating paths (these are contractible because of locality) by induction on the length of the word representing the path. A single-letter path $R_{ij}$ cannot be closed. A two-letter path can only be closed if the word is $R_{ji}R_{ij}$. A three-letter path cannot be closed as can easily be seen considering all possibilities. For example, if we want to try to close the path starting with $R_{23}R_{12}$, we have to add $R_{31}$ on the right so that vertex 1 goes  back to 1 but $R_{31}R_{23}R_{12}$ is not closed since 2 goes to -3 and 3 goes to -2. (We adopt enumeration of the $2n$ vertices with $\{1,2,\dots,n,-1,-2,\dots,-n\}$ where $-i$ denotes the vertex opposite to $i$. 
The element $R_{ij}$ moves vertex $i$ to vertex $j$, vertex $j$ to vertex  $-i$ and leaves all other vertices in place.)
Considering four-letter paths, it is easy to check that they can be closed either if they consist of a product of two closed two-letter paths or if they involve only rotations in a three-dimensional subspace spanned by some three axes $i$, $j$ and $k$.  
These are precisely what we called triangular closed paths in Section 3. In general notations it turns out that all such closed paths are words that are cyclic permutations of the following four expressions:
\begin{equation}\label{triang}
R_{kj}R_{ki}R_{jk}R_{ij},\quad R_{jk}R_{ik}R_{kj}R_{ij},\quad R_{ki}R_{jk}R_{ik}R_{ij},\quad 
R_{ik}R_{kj}R_{ki}R_{ij}.
\end{equation}
Notice that setting these expressions to one gives conjugation identities, e.g. $R_{kj}R_{ik}R_{jk}$ \\ $=R_{ij}$, etc.
Suppose now that the statement we want to prove is valid for all words with length $2n$ and consider a word of $2n+2$ letters. Suppose that the first letter on the right is $R_{ij}$. We can move $R_{ij}$ to the left across any $R_{kl}$ with $k$, $l$ different from $i$ and $j$, as rotations in two such planes commute. If $R_{ij}$ gets next to a word $R_{ji}$ we are allowed to cancel the two and obtain a word of length $2n$ and we are done. The possibility to obtain $R_{ij}R_{ij}$ is excluded by locality since anything to the right leaves vertex $i$ invariant and $R_{ij}^2$ sends $i$ to $-i$. There are four additional possibilities in which $R_{ij}$ gets next to a letter with which it does not commute. These are $R_{jk}R_{ij}$, $R_{kj}R_{ij}$,
$R_{ik}R_{ij}$ and $R_{ki}R_{ij}$. Using the identities following from setting the expressions in Equation (\ref{triang}) to one, we replace the above combinations by products of two letters in which the index $i$ appears only in the letter on the left. Here are the actual identities that can be used:
$$
R_{jk}R_{ij}=R_{ik}R_{jk},\quad R_{kj}R_{ij}=R_{ki}R_{kj},\quad R_{ik}R_{ij}=R_{ij}R_{kj},\quad 
R_{ki}R_{ij}=R_{ij}R_{jk}.
$$
Next, we continue moving the letter involving $i$ to the left. Notice that the index $i$ may now appear as the second index, so we may need to use the additional identities
$$
R_{jk}R_{ji}=R_{ki}R_{jk},\quad R_{kj}R_{ji}=R_{ik}R_{kj},\quad R_{ik}R_{ji}=R_{ji}R_{jk},\quad 
R_{ki}R_{ji}=R_{ji}R_{kj}.
$$
Eventually the letter involving the index $i$ will reach the left end of the word and there will be no letters involving $i$ to its right. Such a path cannot be closed as it does not leave vertex $i$ in place. Therefore, the only possibility is that the letter involving $i$ gets canceled in the process and the length of the word becomes $2n$.\par
Now consider a general closed path $R$ consisting of generating paths and homotopic to the identity. In general the homotopy from the identity to $R$ need not pass through paths consisting of generating paths but we can modify it so that it does. Indeed, if $R(s)$ is the homotopy from the identity to $R$, i.e. $R(0)\equiv Id$, $R(1)=R$, for each $s$ we can homotope $R(s)$ to the nearest closed path consisting of generating paths, as described earlier in this section. For $s$ small enough the nearest will be $Id$. For some $s_1$ the construction will yield a path $R'$ consisting of generating paths, such that $R'$ is homotopic to $R(s_1)$. Furthermore, the construction is such that $R'$ is local closed path as any vertex remains in some $(n-1)$-cell without leaving it. Then for some $s_2>s_1$ we will get $R''$ which is nearest to $R(s_2)$. The path $R''$ is not local but $R'^{-1}R''$ (usual product of paths by concatenation) is local. Since the words giving $R'$ and $R'^{-1}R''$ can be reduced to the identity by inserting words giving triangular closed paths, the same will be true for $R''$. Proceeding in this way, after a finite number of steps we show that the word for the closed path $R$ has the same property.

Acknowledgement \\
We would like to thank Tatiana Gateva-Ivanova for some helpful advice and suggestions.

\end{document}